\documentclass[12pt]{amsart}
\usepackage{graphicx}
\usepackage[cp1250]{inputenc}
\newtheorem{theorem}{Theorem}[section]

\newtheorem{example}[theorem]{Example}
\newtheorem{remark}[theorem]{Remark}
\newtheorem{corollary}[theorem]{Corollary}
\newtheorem{lemma}[theorem]{Lemma}

\renewcommand{\Re}{\mathop{\rm Re}}
\renewcommand{\Im}{\mathop{\rm Im}}

\begin{document}
\title{A generalization of the Clunie--Sheil-Small theorem}
\author{Ma{\l}gorzata Michalska, Andrzej M. Michalski}

\address{
Ma{\l}gorzata Michalska,  \newline Institute of Mathematics,
\newline Maria Curie-Sk{\l}odowska University, \newline pl. M.
Curie-Sk{\l}odowskiej 1, \newline 20-031 Lublin, Poland}
\email{malgorzata.michalska@poczta.umcs.lublin.pl}

\address{
Andrzej M. Michalski, \newline
Department of Complex Analysis, \newline
The John Paul II Catholic University of Lublin, \newline
ul. Konstantyn\'{o}w 1H, \newline
20-950 Lublin, Poland}
\email{amichal@kul.lublin.pl}

\date{\today}
\subjclass[2010]{31A05, 30C55, 30C45} \keywords{harmonic mappings,
convex in one direction, shear construction} \maketitle
\begin{abstract}
In 1984, a simple and useful univalence criterion for harmonic
functions was given by Clunie and Sheil-Small, which is usually
called the shear construction. However, the application of this
theorem is limited to the planar harmonic mappings convex in the
horizontal direction. In this paper, a natural generalization of the
shear construction is given. More precisely, our results are
obtained under the hypothesis that the image of a harmonic mapping
is a sum of two sets convex in the horizontal direction.
\end{abstract}

\baselineskip1.4\baselineskip

\section{Introduction}
Let $\mathbb{D}:=\{z\in\mathbb{C}:|z|<1\}$ be the open unit disk in
the complex plane $\mathbb{C}$.
A function $f:\mathbb{D}\to\mathbb{C}$ is said to be harmonic,
if its real and imaginary parts are real harmonic, i.e. they satisfy the Laplace equation.
Since $\mathbb{D}$ is simply connected it is well-known that $f$ can be written in the form
\begin{equation}\label{f_sum_hol}
  f(z)=h(z)+\overline{g(z)},\quad z\in\mathbb{D},
\end{equation}
where $h$ and $g$ are analytic in $\mathbb{D}$.
The Jacobian $J_f$ of $f$ in terms of $h$ and $g$ is given by
\begin{equation}\label{f_jacobian}
  J_f(z)=|h'(z)|^2-|g'(z)|^2,\quad z\in\mathbb{D}.
\end{equation}
Among all the harmonic functions in $\mathbb{D}$ one can distinguish
those with non-vanishing Jacobian. In fact, it is proved that such
harmonic functions are locally 1-1. If the Jacobian of a harmonic
function in $\mathbb{D}$ is positive, it means that this function is
locally 1-1 and sense-preserving. More information about basics of
harmonic functions can be found e.g. in \cite{Duren1}.

Clunie and Sheil-Small in \cite{ClunieSheilSmall1} gave the following theorem, known as the shear construction.

{\renewcommand{\thetheorem}{A} 
\begin{theorem}\label{shear_construction}
A function $f=h+\overline{g}$ harmonic in $\mathbb{D}$ with positive Jacobian is 1-1 sense-preserving mapping of $\mathbb{D}$
onto a domain convex in the direction of the real axis if, and only if, $h-g$ is an analytic 1-1 mapping of $\mathbb{D}$
onto a domain convex in the direction of the real axis.
\end{theorem}
\addtocounter{theorem}{-1}}%

It appeared to have many applications as an univalence criterion and
as a method of constructing harmonic mappings (see, e.g.,
\cite{DorffNowakWoloszkiewicz, DorffSzynal, DriverDuren,
GanczarWidomski, GrigorianSzapiel, HengartnerSchober,
KlimekSmetMichalski, Livingston}).

In this paper we generalize the theorem of Clunie and Sheil-Small.
In Section 2 we show some auxiliary results. In Section 3 we use
results from Section 2 to give new conditions for univalence of the
planar harmonic mappings.

\section{Topological properties}\setcounter{equation}{0}
The proof of Theorem \ref{shear_construction} of Clunie and Sheil-Small relies on the following lemma,
which will be also useful in our considerations.

{\renewcommand{\thetheorem}{B} 
\begin{lemma}\label{shear_lemma}
Let $D$ be a domain convex in the direction of the real axis and let $p$ be a continuous real-valued function in $D$.
Then the mapping $D\ni w\mapsto w+p(w)$ is 1-1 in $D$ if, and only if, it is locally 1-1.
In this case the image of $D$ is convex in the direction of the real axis.
\end{lemma}
\addtocounter{theorem}{-1}}%

Using this lemma we will prove more general results and apply them to obtain new univalence criteria
for harmonic mappings.
For a given set $D$ in the complex plane $\mathbb{C}$ it will be convenient to define
\begin{equation}\label{set_projection}
  P_y(D):=\left\{a\in\mathbb{R}:\exists_{z\in D}\Im z=a\right\}.
\end{equation}
Such defined set $P_y(D)$ has several immediate properties,
which we formulate in the following lemma for convenience in further use.

\begin{lemma}\label{projection_lemma}
Let $D_1$ and $D_2$ be the domains with nonempty intersection such
that $D_1\cup D_2$ is simply connected. Then $P_y(D_1)$, $P_y(D_2)$
and $P_y(D_1\cap D_2)$ are open intervals.
\end{lemma}
\begin{proof}
The Janiszewski theorem \cite[p.~268, Theorem~2]{Kuratowski1} yields the connectedness of the set $D_1\cap D_2$,
which clearly is also open. Thus $D_1\cap D_2$ is a nonempty domain as well as $D_1$ and $D_2$.
Hence, obviously, $P_y(D_1)$, $P_y(D_2)$ and $P_y(D_1\cap D_2)$ are open and connected subset of the real line $\mathbb{R}$,
which completes the proof.
\end{proof}
Using this lemma we can prove the following theorem.
\begin{theorem}\label{thm_1}
Let $D_1$, $D_2$ be the domains convex in the direction of the real axis
and let $q:D_1\cup D_2\to\mathbb{C}$ be a continuous function for which Jacobian $J_q$ exists, and such that $\Im q(z)=\Im z$ for all $z\in D_1\cup D_2$.
Then $q$ is 1-1 if, and only if, $J_q\neq 0$ and
  $$P_y(D_1\cap D_2)=P_y(q(D_1)\cap q(D_2)).$$
\end{theorem}
\begin{proof}
If $D_1$, $D_2$ are two disjoint domains convex in the direction of
the real axis then our claim follows immediately from Theorem
\ref{shear_construction}. Hence, we consider the case $D_1\cap
D_2\not=\emptyset$.

Assume that $q$ is 1-1 in $D_1\cup D_2$. We show that Jacobian is
not equal to $0$ and $P_y(D_1\cap D_2)=P_y(q(D_1)\cap q(D_2))$. It
is clear that if $q$ is 1-1, then it is locally 1-1 and thus
$J_q\neq 0$. It is also clear that $P_y(D_1\cap D_2)\subset
P_y(q(D_1)\cap q(D_2))$. We show the inverse inclusion. Let $a\in
\mathbb{R}\setminus P_y(D_1\cap D_2)$ be fixed. Then for any choice
of $z_1\in D_1$ and $z_2\in D_2$ such that $\Im z_1=a$ and $\Im
z_2=a$ we have $q(z_1)\neq q(z_2)$, since $q$ is 1-1. Thus we deduce
that $a\notin P_y(q(D_1)\cap q(D_2))$, which means that
$P_y(q(D_1)\cap q(D_2))\subset P_y(D_1\cap D_2)$. Hence we get
$P_y(D_1\cap D_2)=P_y(q(D_1)\cap q(D_2))$.

To prove the converse we assume that $J_q\neq 0$ and $P_y(D_1\cap D_2)=P_y(q(D_1)\cap q(D_2))$ and we show that $q$ is 1-1.
The property $\Im q(z)=\Im z$ for all $z\in D_1\cup D_2$ together with Lemma \ref{shear_lemma} ensure that $q$ is 1-1 in $(D_1\cup D_2)\cap\{z\in\mathbb{C}:\Im z\in P_y(D_1\cap D_2)\}$.
Assume that $q$ be not 1-1 in
  $$\widetilde{D}:=(D_1\cup D_2)\cap\{z\in\mathbb{C}:\Im z\notin P_y(D_1\cap D_2)\}.$$
Then, there exist $a\in \widetilde{D}$ and $z_1,z_2\in D_1\cup D_2$ such that $a=\Im z_1=\Im z_2$ and $q(z_1)=q(z_2)$.
But the last equality means that $a\in P_y(q(D_1)\cap q(D_2))$ and by the definition of $\widetilde{D}$ we have $a\notin P_y(D_1\cap D_2)$, which is a contradiction to the assumption that $P_y(D_1\cap D_2)=P_y(q(D_1)\cap q(D_2))$. Thus, $q$ is 1-1 in $\widetilde{D}$. Now, the property $\Im q(z)=\Im z$ for all $z\in D_1\cup D_2$ implies that $q$ is 1-1 in $D_1\cup D_2$ and this completes the proof.
\end{proof}

Replacing the univalence condition in Theorem \ref{thm_1} by the condition that the sets $D_1\cup D_2$ and $q(D_1)\cup q(D_2)$ are simply connected we get.

\begin{theorem}\label{thm_2}
Let $D_1$, $D_2$ be the domains convex in the direction of the real
axis with nonempty intersection and let $q:D_1\cup D_2\to\mathbb{C}$
be a continuous function, such that $J_q$ exists and it is not equal
to $0$, and $\Im q(z)=\Im z$ for all $z\in D_1\cup D_2$. If $D_1\cup
D_2$ and $q(D_1)\cup q(D_2)$ are simply connected then
  $$P_y(D_1\cap D_2)=P_y(q(D_1)\cap q(D_2)).$$
\end{theorem}

\begin{proof}
First, observe that the inclusion
\begin{equation}\label{incl}
  P_y(D_1\cap D_2)\subset P_y(q(D_1)\cap q(D_2))
\end{equation}
is valid for all domains $D_1$, $D_2$ and for all functions $q$ satisfying assumptions of Theorem \ref{thm_2}.

Now, we prove the inverse inclusion. We can assume that the Jacobian $J_q$ is positive.
Notice, that if $D_1\cup D_2$ is simply connected
then, by Lemma \ref{projection_lemma}, the set $P_y(D_1\cap D_2)$ is connected.
By the same Lemma \ref{projection_lemma} and obvious equality $q(D_1\cup D_2)=q(D_1)\cup q(D_2)$
we deduce that $P_y(q(D_1)\cap q(D_2))$ is connected since $q(D_1\cup D_2)$ is simply connected.
Moreover, $P_y(D_1\cap D_2)$ and $P_y(q(D_1)\cap q(D_2))$ are open since $D_1\cap D_2$ and $q(D_1)\cap q(D_2)$ are open, hence $P_y(D_1\cap D_2)$ and $P_y(q(D_1)\cap q(D_2))$ are nonempty open intervals.

Next, assume that there exists a real number $a$ such that $a\in P_y(q(D_1)\cap q(D_2))$ and $a\notin P_y(D_1\cap D_2)$.
Then, there exist $\tilde{a}\in P_y(q(D_1)\cap q(D_2))\setminus P_y(D_1\cap D_2)$
and $\varepsilon>0$ such that the sets
  $$\widetilde{A}_\varepsilon:=(\tilde{a}-\varepsilon,\tilde{a}+\varepsilon)\cap P_y(D_1\cap D_2)
  \quad\text{ and }\quad
  (\tilde{a}-\varepsilon,\tilde{a}+\varepsilon)\setminus (P_y(D_1\cap D_2)\cup\{\tilde{a}\})$$
are nonempty open intervals. Indeed, this follows from the properties of $P_y(D_1\cap D_2)$ and $P_y(q(D_1)\cap q(D_2))$ as the open and nonempty intervals. 
Now, since $\tilde{a}\in P_y(q(D_1)\cap q(D_2))$ and $q(D_1)\cap q(D_2)$ is open, we can find points $w_1,w_2\in q(D_1)\cap q(D_2)$
such that
  $$\Re w_1<\Re w_2 \quad\text{ and }\quad \Im w_1=\Im w_2=\tilde{a}.$$
Recall, that $\tilde{a}\notin P_y(D_1\cap D_2)$, thus there exist points
  $$\eta_1,\eta_2\in D_1
  \quad \text{ and }\quad \zeta_1,\zeta_2\in D_2
  $$
such that $q(\eta_1)=q(\zeta_1)=w_1$ and $q(\eta_2)=q(\zeta_2)=w_2$ and, by Lemma \ref{shear_lemma}, they are unique.
Moreover, the assumption that the Jacobian $J_q$ is positive implies either
\begin{eqnarray}
  && \Re \eta_1<\Re \eta_2 < \Re \zeta_1<\Re \zeta_2,\nonumber \\
  \text{or}&& \nonumber\\
  && \Re \zeta_1<\Re \zeta_2 < \Re \eta_1<\Re \eta_2. \nonumber
\end{eqnarray}
Now, since $D_1$ and $D_2$ are open sets and $\widetilde{A}_\varepsilon$ is a nonempty, open interval,
then there exist sequences
  $$\mathbb{N}\ni n\mapsto \eta_{1,n}\in D_1
  ,\ \eta_{1,n}\to\eta_1, \quad\text{ and }\quad \mathbb{N}\ni n\mapsto \eta_{2,n}\in D_1
  ,\ \eta_{2,n}\to\eta_2,$$
and the sequences
  $$\mathbb{N}\ni n\mapsto \zeta_{1,n}\in D_2
  ,\ \zeta_{1,n}\to\zeta_1 \quad\text{ and }\quad \mathbb{N}\ni n\mapsto \zeta_{2,n}\in D_2
  ,\ \zeta_{2,n}\to\zeta_2,$$
with $\Im\eta_{1,n}=\Im\eta_{2,n}=\Im\zeta_{1,n}=\Im\zeta_{2,n}\in \widetilde{A}_{\varepsilon}$ and such that
either
\begin{eqnarray}\label{real_part_ineq_sq}
  && \Re \eta_{1,n}<\Re \eta_{2,n} < \Re \zeta_{1,n}<\Re \zeta_{2,n},\nonumber \\
  \text{or}&& \\
  && \Re \zeta_{1,n}<\Re \zeta_{2,n} < \Re \eta_{1,n}<\Re \eta_{2,n},\nonumber
\end{eqnarray}
for sufficiently large $n$. Next, from the continuity of $q$ we deduce that
  $$q(\eta_{1,n})\to w_1, \qquad q(\eta_{2,n})\to w_2 \quad \text{ and } \quad q(\zeta_{1,n})\to w_1,\qquad q(\zeta_{2,n})\to w_2.$$
Thus, by the assumption that $J_q>0$ and \eqref{real_part_ineq_sq} we have either
\begin{eqnarray*}
  &&
  \Re q(\eta_{2,n})
  <
  \Re q(\zeta_{1,n})
  \quad\text{or}\quad
  \Re q(\zeta_{2,n})
  <
  \Re q(\eta_{1,n}),
\end{eqnarray*}
for sufficiently large $n$, which implies $
\Re w_2
\leq
\Re w_1
$.
But this is a contradiction to the assumption $\Re w_1<\Re w_2$.
Thus, we have the inclusion $P_y(q(D_1)\cap q(D_2))\subset P_y(D_1\cap D_2)$,
which together with \eqref{incl} yields $P_y(D_1\cap D_2)=P_y(q(D_1)\cap q(D_2))$, and this completes the proof.
\end{proof}

\begin{corollary}\label{cor_1}
Let $D_1$, $D_2$ be the domains convex in the direction of the real axis with nonempty intersection, such that
$D_1\cup D_2$ is simply connected
and let $q:D_1\cup D_2\to\mathbb{C}$ be a continuous function for which the Jacobian $J_q$ exists, such that $\Im q(z)=\Im z$ for all $z\in D_1\cup D_2$
and $q(D_1)\cup q(D_2)$ is simply connected. Then $J_q\neq 0$ if, and only if, $q$ is 1-1.
\end{corollary}
\begin{proof}
It is an immediate consequence of Theorem \ref{thm_1} and Theorem \ref{thm_2}.
\end{proof}

\section{Harmonic mappings}\setcounter{equation}{0}
In this section we apply the results obtained in the previous section to the theory of harmonic mappings.
We start with the definition which will simplify our considerations.
For a given set $D$ let
\begin{equation}\label{set_B}
  \Lambda_y(D):=\left\{a\in\mathbb{R}:(D\cap \{z\in\mathbb{C}:\Im z=a\})\ \text{is a nonempty and connected set}\right\}.
\end{equation}
We will see the set $\Lambda_y$ is as much convenient in the
following investigations as the set $P_y$, defined by
\eqref{set_projection}, was in the previous section. Thus, we need
the following lemma describing a connection between $P_y$ and
$\Lambda_y$.
\begin{lemma}\label{A_B_conection}
Let $D_1$, $D_2$ be the domains convex in the direction of the real
axis with nonempty intersection. Then $P_y(D_1\cap
D_2)=\Lambda_y(D_1\cup D_2)$.
\end{lemma}
\begin{proof}
Let $D_1$, $D_2$ be the domains convex in the direction of the real axis with nonempty intersection.
We will show  both inclusions
  $$P_y(D_1\cap D_2)\subset\Lambda_y(D_1\cup D_2)\quad \text{ and } \quad \Lambda_y(D_1\cup D_2)\subset P_y(D_1\cap D_2).$$

Assume first, that $a\in P_y(D_1\cap D_2)$. Then, there exists $w\in D_1\cap D_2$ such that $\Im w=a$.
This means, that
  $$w\in (D_1\cap\{z\in\mathbb{C}:\Im z=a\})\cap(D_2\cap\{z\in\mathbb{C}:\Im z=a\}).$$
Next, observe that the sets $D_1 \cap\{z\in\mathbb{C}:\Im z=a\}$ and $D_2 \cap\{z\in\mathbb{C}:\Im z=a\}$ are nonempty and connected,
since both domains $D_1$ and $D_2$ are convex in the direction of the real axis, and in addition they have nonempty intersection.
Thus, the set
  $$(D_1\cup D_2)\cap\{z\in\mathbb{C}:\Im z=a\}$$
is nonempty and connected, and consequently $a\in \Lambda_y(D_1\cup
D_2)$.

Now, we prove the second inclusion. Let $a\in \Lambda_y(D_1\cup
D_2)$, then the set
  $$(D_1\cup D_2)\cap\{z\in\mathbb{C}:\Im z=a\}$$
is nonempty and connected.
Next, observe that
  $$D_1\cap\{z\in\mathbb{C}:\Im z=a\}\quad \text{ and } \quad D_2\cap\{z\in\mathbb{C}:\Im z=a\}$$ are open and connected intervals
since $D_1$ and $D_2$ are open and convex in the direction of the real axis.
Hence, there exists $w\in D_1\cap D_2$, such that $\Im w=a$ and thus, $a\in P_y(D_1\cap D_2)$, which completes the prove.
\end{proof}

Now, we can apply results obtained in Section 2 to harmonic mappings.
\begin{theorem}\label{thm_3}
Let $f=h+\overline{g}$ be a harmonic function in $\mathbb{D}$ such
that $J_f>0$ in $\mathbb{D}$. If
$\Lambda_y((h-g)(\mathbb{D}))=\Lambda_y(f(\mathbb{D}))$ then the
following statements are equivalent
\begin{enumerate}
\item[(1)]{$f$ is 1-1
    mapping and $f(\mathbb{D})$ is a sum of two non-disjoint domains
    convex in the direction of the real axis.}
\item[(2)]{$h-g$ is 1-1 analytic mapping and $(h-g)(\mathbb{D})$ is a sum of two non-disjoint domains
    convex in the direction of the real axis.}
\end{enumerate}
\end{theorem}
\begin{proof}
Let $f=h+\overline{g}$ be a harmonic function in the unit disk and
such that $J_f$ is positive in $\mathbb{D}$, and let
$\Lambda_y((h-g)(\mathbb{D}))=\Lambda_y(f(\mathbb{D}))$. We show
that $(1)=>(2)$ and $(2)=>(1)$.

$(1)=>(2)$. Assume that $f$ is 1-1 in the unit disk and that $f(\mathbb{D})=D_1\cup D_2$,
where $D_1,D_2\subset\mathbb{C}$ are domains convex in the direction of the real axis with nonempty intersection.
Then there exists $f^{-1}:D_1\cup D_2\to \mathbb{D}$ and the composition $q:=(h-g)\circ f^{-1}$ is well defined continuous function.
Observe, that $q(w)=(h-g)(f^{-1}(w))=w-2\Re g(f^{-1}(w))$ for all $w\in D_1\cup D_2$.
Moreover, by Lemma \ref{A_B_conection} we have
\begin{equation}\label{proj_f}
\Lambda_y(f(\mathbb{D}))=\Lambda_y(D_1\cup D_2)=P_y(D_1\cap D_2).
\end{equation}
Similarly, by Lemma \ref{A_B_conection} and by equality $q(D_1\cup D_2)=q(D_1)\cup q(D_2)$ we have
\begin{equation}\label{proj_q}
\Lambda_y((h-g)(\mathbb{D}))
=\Lambda_y(q(D_1\cup D_2))=\Lambda_y(q(D_1)\cup
q(D_2))=P_y(q(D_1)\cap q(D_2)).
\end{equation}
The formulae \eqref{proj_f} and \eqref{proj_q}, together with the
hypothesis $\Lambda_y((h-g)(\mathbb{D}))=\Lambda_y(f(\mathbb{D}))$,
yield
\begin{equation}\label{proj_f_q}
   P_y(D_1\cap D_2)=P_y(q(D_1)\cap q(D_2)).
\end{equation}
Thus, the assumptions of Theorem \ref{thm_1} are satisfied and in consequence we obtain that $q$ is 1-1 in $\mathbb{D}$.
Hence, $h-g$ is 1-1 in $\mathbb{D}$, since $f$ is. Additionally, both sets $q(D_1)$ and $q(D_2)$ are domains convex in the direction of the real axis, by Lemma \ref{shear_lemma}, and their intersection is not empty by \eqref{proj_f_q}.

$(2)=>(1)$. Now, assume that $h-g$ is 1-1 in the unit disk and that $(h-g)(\mathbb{D})=\Omega_1\cup \Omega_2$,
where $\Omega_1,\Omega_2\subset\mathbb{C}$ are domains convex in the direction of the real axis with nonempty intersection.
Then there exists $(h-g)^{-1}:\Omega_1\cup \Omega_2\to \mathbb{D}$ and the composition $q:=f\circ (h-g)^{-1}$ is well defined continuous function.
Observe, that we have $q(w)=f((h-g)^{-1}(w))=w+2\Re g((h-g)^{-1}(w))$ for all $w\in \Omega_1\cup \Omega_2$.
Reasoning similar to the one used in previous case and the use of Lemma \ref{A_B_conection} give us equality
\begin{equation}\label{proj_om_f_q}
   P_y(\Omega_1\cap \Omega_2)=P_y(q(\Omega_1)\cap q(\Omega_2)).
\end{equation}
Again, the assumptions of Theorem \ref{thm_1} are satisfied and in consequence we obtain that $q$ is 1-1 in $\mathbb{D}$,
thus $f$ is 1-1 in $\mathbb{D}$, since $h-g$ is. Finally, $q(\Omega_1)$ and $q(\Omega_2)$ are domains convex in the direction of the real axis, by Lemma \ref{shear_lemma}, and their intersection is not empty by \eqref{proj_om_f_q}.
\end{proof}

As a consequence of Theorem \ref{thm_3} we obtain a generalization
of Theorem \ref{shear_construction} of Clunie and Sheil-Small.
\begin{theorem}\label{thm_4}
Let $f=h+\overline{g}$ be a harmonic function in $\mathbb{D}$ such that $J_f>0$ in $\mathbb{D}$.
If $(h-g)(\mathbb{D})$ and $f(\mathbb{D})$ are nonempty simply connected domains then the following statements are equivalent
\begin{enumerate}
\item[(1)]{$f$ is 1-1 mapping and $f(\mathbb{D})$ is a sum of two non-disjoint domains
                 convex in the direction of the real axis.}
\item[(2)]{$h-g$ is 1-1 analytic mapping and $(h-g)(\mathbb{D})$ is a sum of two non-disjoint domains
                 convex in the direction of the real axis.}
\end{enumerate}
\end{theorem}
\begin{proof}
Observe, that if $f$ is 1-1 in $\mathbb{D}$ and $f(\mathbb{D})=D_1\cup D_2$,
where $D_1,D_2\subset\mathbb{C}$ are domains convex in the direction of the real axis with nonempty intersection,
then the function
  $$D_1\cup D_2\ni w\mapsto q_{f}(w):=(h-g)(f^{-1}(w))=w-2\Re g(f^{-1}(w))$$
is well-defined and continuous in $D_1\cup D_2$.
The same is true if we assume that $h-g$ is 1-1 in $\mathbb{D}$ and $(h-g)(\mathbb{D})=\Omega_1\cup \Omega_2$,
where $\Omega_1,\Omega_2\subset\mathbb{C}$ are domains convex in the direction of the real axis with nonempty intersection, that is
the function
  $$D_1\cup D_2\ni w\mapsto q_{h-g}(w):=f((h-g)^{-1}(w))=w+2\Re g((h-g)^{-1}(w))$$
is well-defined and continuous in $\Omega_1\cup \Omega_2$.

Since $(h-g)(\mathbb{D})$ and $f(\mathbb{D})$ are nonempty simply connected domains then by Theorem \ref{thm_2} and Lemma \ref{A_B_conection}, the proof follows from Theorem \ref{thm_3}.
\end{proof}

If one omits in Theorem \ref{thm_4} the assumption that both $f(\mathbb{D})$ and $(h-g)(\mathbb{D})$ are simply connected, then the Theorem \ref{thm_4} is no longer true which is shown in the following example.

\begin{example}\rm
Consider vertical shear of the rotated Koebe function with
dilatation $\omega(z):=iz$. From the equations
\begin{align}
  &h(z)-g(z)=\frac{z}{(1-iz)^2}\nonumber \\
  &g^\prime(z)=izh^\prime(z)\nonumber
\end{align}
we get
\begin{align}
  &h(z)=\frac{-6iz-3z^2+iz^3}{6(i+z)^3},\nonumber \\
  &g(z)=\frac{3z^2+iz^3}{6(i+z)^3},\nonumber
\end{align}
and
  $$f(z)=h(z)+\overline{g(z)}=\frac{-6iz-3z^2+iz^3}{6(i+z)^3}+\overline{\left(\frac{3z^2+iz^3}{6(i+z)^3}\right)}.$$
Now, using transformation
  $$w=u+iv:=\frac{1+iz}{1-iz},$$
which maps the unit disk onto the right half-plane, i.e.
$\{w\in\mathbb{C}:\Re w>0\}$ we get
\begin{align}
  &h(z)-g(z)=\frac{1}{4i}(w^2-1),\nonumber \\
  &h(z)+g(z)=\frac{1}{6i}(w^3-1),\nonumber
\end{align}
and consequently
  $$f(z)=\Re(h(z)+g(z))+i\Im(h(z)-g(z))=-\frac{1}{6}\Im(w^3-1)-\frac{i}{4}\Re(w^2-1).$$
After some calculations we obtain
\begin{equation}
\label{map_f}
  f(z)=-\frac{1}{6}v(3u^2-v^2)-\frac{i}{4}(u^2-v^2-1),
\end{equation}
where $u>0$ and $v\in \mathbb{R}$.

Clearly, the function $h(z)-g(z)$ maps the unit disk onto the plane
with the slit along the imaginary axis, more precisely onto
$\mathbb{C}\setminus \{z\in\mathbb{C}: \Im z\geq \frac{1}{4}\text{
and }\Re z=0\}$, which is a simply connected domain. On the other
hand, the formula \eqref{map_f}, allows us to find the image of the
unit disk via the map $f(z)$, by studying which parts of the
vertical lines of the complex plane belong to $f(\mathbb{D})$.
First, observe that $\Re f(z)=0$ if and only if $v=0$ or $v^2=3u^2$.
Thus, we have $\Im f(z)=\frac{1}{4}-\frac{u^2}{4}$, with $u>0$, if
$v=0$ and $\Im f(z)=\frac{1}{4}+\frac{u^2}{2}$, with $u>0$, if
$v^2=3u^2$, and consequently we get that the point $\frac{i}{4}$ do
not belong to $f(\mathbb{D})$.

Now, assume that $\Re f(z)=c$ with $c\neq 0$. Then, since $v\neq 0$,
we have $u^2=\frac{v^2}{3}-\frac{2c}{v}$ and
  $$\Im f(z)=\frac{2v^3+3v+6c}{12v}, \quad\text{where } v\in(-\infty,0)\cup(0,+\infty).$$
If $c>0$ and $v\in (-\infty,0)$ then
\begin{align}
  &\lim_{v\to -\infty} \frac{2v^3+3v+6c}{12v}=+\infty,\nonumber\\
  &\lim_{v\to 0^-} \frac{2v^3+3v+6c}{12v}=-\infty,\nonumber
\end{align}
and the whole vertical line $w=c$ belongs to $f(\mathbb{D})$. Analogously, if $c<0$ and $v\in (0,+\infty)$ then
\begin{align}
  &\lim_{v\to +\infty} \frac{2v^3+3v+6c}{12v}=+\infty,\nonumber\\
  &\lim_{v\to 0^+} \frac{2v^3+3v+6c}{12v}=-\infty,\nonumber
\end{align}
and in that case the whole vertical line $w=c$ belongs to
$f(\mathbb{D})$, too. Hence, we get
$f(\mathbb{D})=\mathbb{C}\setminus\{\frac{i}{4}\}$ which is not
simply connected domain. The function $f$ fail to satisfy
assumptions of Theorem \ref{thm_3}, and straightforward calculations
shows that $f(\frac{\sqrt{3}}{3})=f(-\frac{\sqrt{3}}{3})
=\frac{3i}{8}$, thus $f$ is not univalent in $\mathbb{D}$.
\end{example}

\begin{remark}
Recall, that Theorem \ref{shear_construction} can be reformulated
and it remains valid for a function convex in any fixed direction.
Notice, that our results can also be rewritten in this fashion.
\end{remark}

\end{document}